\documentclass[11pt]{amsart}

\usepackage{amsmath, amssymb, amsfonts,enumerate}
\usepackage[all]{xy}
\usepackage{amscd}
\usepackage{comment}
\usepackage{mathtools}
\usepackage{tikz} 
\usepackage{tikz-cd} 
\tikzcdset{diagrams={nodes={inner sep=7.5pt}}}

\usepackage{url}
\usepackage{hyperref}

\newtheorem{theorem}{Theorem}[section]
\newtheorem{lemma}[theorem]{Lemma}
\newtheorem{corollary}[theorem]{Corollary}

\newtheorem{theoremletter}{Theorem}

 \theoremstyle{definition}

\newtheorem*{Claim}{Claim}

  \newtheorem*{example*}{Example}

\numberwithin{equation}{section}
 
\newcommand {\N}{\mathbb{N}} 
\newcommand {\Z}{\mathbb{Z}}

\newcommand{\PP}{\mathcal{P}}

\newcommand{\A}{\mathbb{A}}

\DeclareMathOperator{\End}{End}

\DeclareMathOperator{\Id}{Id}

\DeclareMathOperator{\Spec}{Spec}

\begin{document}

\title[Geometric Kaplansky's direct finiteness conjecture]{A geometric generalization of  Kaplansky's direct finiteness conjecture}
\author[Xuan Kien Phung]{Xuan Kien Phung}
\email{phungxuankien1@gmail.com}
\subjclass[2010]{14A10, 14A15, 16S34, 20C07, 37B10, 68Q80}
\keywords{group ring, near ring, direct finiteness, stable finiteness, sofic group, surjunctivity, algebraic variety, algebraic group, cellular automata} 
\begin{abstract} 
Let $G$ be a group and let $k$ be a field. Kaplansky's direct finiteness conjecture states that every one-sided unit of the group ring $k[G]$ must be a two-sided unit. In this paper,  
we establish a geometric direct finiteness theorem for endomorphisms of symbolic algebraic varieties.  Whenever $G$ is a sofic group or more generally a surjunctive group, our result implies a generalization of Kaplansky's direct finiteness conjecture for the near ring $R(k,G)$ which is  $k[X_g\colon g \in G]$ as a group and which contains naturally $k[G]$ as the subring of homogeneous polynomials of degree one. We also prove that  Kaplansky's stable finiteness conjecture is a consequence of Gottschalk's Surjunctivity conjecture. 
\end{abstract} 
\date{\today}
\maketitle

\setcounter{tocdepth}{1}

\section{Introduction}
In \cite{kap}, 
Kaplansky
conjectured that for every group $G$ and for  every field $k$, the group ring $k[G]$ is directly finite, i.e., for $a,b \in k[G]$ with  $ab=1$, one has $ba=1$. 
The conjecture was settled by 
Kaplansky himself in \cite{kap} in characteristic zero but is still open in positive characteristic. However, in arbitrary characteristic, the conjecture is known for the wide class of sofic groups $G$ (see \cite{gromov-esav},  \cite{weiss-sgds}) while no examples of non sofic groups are known in the literature. 
Results in  \cite{elek} proved that $k[G]$ is directly finite
for every field $k$ and every sofic group $G$. More generally, the papers \cite{elek} (see also \cite{ara}), \cite{csc-artinian}, and  \cite{li-liang} show that  $k[G]$ is directly finite for every sofic group
$G$ when $k$ is respectively a division ring, an Artinian
ring, a left Noetherian ring. 
\par 
Our first goal is to establish an algebraic generalization of Kaplansky's direct finiteness conjecture for the class of near rings $R(K,G)$ associated with a field $K$ and a sofic group $G$ that we describe briefly in the sequel. 
\par 
Let $k$ be a field and let $G$ be a group. Following \cite{phung-2020}, we have a near ring $(R(k,G), +, \star)$ where   $(R(k,G),+)=(k[X_g: g\in G],+)$ as a group while the multiplication $\star$ is induced naturally by the group $G$ as follows. 
\par 
Let us consider the semi-ring: 
\begin{equation}
\N[G] \coloneqq \{ f\colon G \to \N \colon f(g) \neq 0 \text{ for finitely many } g\in G \}. 
\end{equation}
\par 
With each $u\in \N[G]$, we associate a monomial 
$
X^u \coloneqq \prod_{g \in G} X_g^{u(g)}$ with the convention $X_g^0 =1$ so that 
\begin{equation} 
R(K,G)=\left\{\sum_{u \in \N[G]} \alpha(u) X^{u} \colon \alpha(u) \neq 0 \text{ for finitely many } u\in \N[G]\right\}.
\end{equation}
\par 
We now describe the natural actions of $G$ on $\N[G]$ and $R(k,G)$.  
Let us fix $g\in G$. For $u \in \N[G]$, we define $gu \in \N[G]$ by setting  
$(gu)(h)\coloneqq u(g^{-1}h)$ for every $h \in G$. Then for  $\gamma=\sum_{w \in \N[G]} \gamma(w) X^{w} \in R(k,G)$, we set  
\begin{equation} 
g\gamma \coloneqq \sum_{w \in \N[G]} \gamma(w) X^{gw} \in R(k,G).
\end{equation} 
\par 

For   
$\alpha=\sum_{u \in \N[G]} \alpha(u) X^{u}$ and 
$\beta=\sum_{v \in \N[G]} \beta(v) X^{v}$  elements of $R(K,G)$, 
their multiplication $\alpha \star \beta$   is given by: 
\begin{align}
\label{e:def-nr}
\alpha \star \beta & \coloneqq \sum_{u \in \N[G]} \alpha(u) X^u\star\beta  \coloneqq \sum_{u \in \N[G]} \alpha(u) \prod_{g \in G} (g\beta)^{u(g)}. 
\end{align}
\par 
Therefore,  
$\alpha \star \beta = 
 \sum_{u \in \N[G]} \alpha(u) \prod_{g \in G} \left(\sum_{v \in  \N[G]}\beta(v) X^{gv} \right)^{u(g)}$. 
 \par 
 
\begin{example*}
Let $g,h,s,t \in G$ and 
$\alpha = X_gX_{h}^2$ + 1, $\beta = X_{s}^2 - X_{t}^3$. 
Then   
\[
\alpha \star \beta = (X_{gs}^2 - X_{gt}^3) (X_{hs}^2 - X_{ht}^3)^2+1, \, \, 
\beta \star \alpha = (X_{sg}X_{sh}^2+1)^2 - (X_{tg}X_{th}^2+1)^3. 
\]
\end{example*}

It follows from \cite[Proposition~10.4]{phung-2020} that  
 $(R(k,G), +, \star )$ is a left near ring, i.e., 
$(R(k,G), \star )$ is a monoid with identity element $X_{1_G}$, and we have 
$
(\alpha +\beta) \star \gamma = \alpha \star \gamma + \beta \star \gamma$ 
for all $\alpha, \beta, \gamma \in R(k,G)$. 
\par 
The first main result of the paper is the following  generalization of Kaplansky's direct finiteness conjecture for the near ring $R(k,G)$ over sofic groups in arbitrary characteristic (see  Theorem~\ref{t:kap-near-ring-surjunctive-group}):  

\begin{theoremletter}
\label{t:kap-near-ring-direct-finite}
Let $G$ be a sofic group and let $k$ be a field. Suppose that  
 $\alpha, \beta \in R(k,G)$ verify  $\alpha \star \beta= X_{1_G}$. Then one has  $\beta \star \alpha=X_{1_G}$. 
\end{theoremletter}
\par 
In particular,  Theorem~\ref{t:kap-near-ring-direct-finite} generalizes \cite[Theorem~10.12]{phung-2020} where $G$ is required to be a residually finite group. 
\par 
On the other hand,  
by \cite[Proposition~10.5]{phung-2020}, it is known that 
for every field $k$ and for every group $G$,  
the canonical map 
$\Phi \colon k[G] \to R(k,G)$ determined  by $ \Sigma_{g \in G} \alpha(g)g \longmapsto  \Sigma_{g \in G} \alpha(g) X_g$ is an embedding of near rings. \par 
As a consequence, we deduce: 

\begin{corollary}
\label{c:sofic-direct-intro}
Let $G$ be a sofic group and let $k$ be a field. Suppose that $\alpha, \beta \in k[G]$ satisfy $\alpha \beta=1$. Then one has $\beta \alpha=1$. \qed 
\end{corollary}

\par 
Now fix a set $A$ called the \emph{alphabet},  and a group  $G$, the \emph{universe}.
A \emph{configuration} $c \in A^G$ is a map $c \colon G \to A$. 
The Bernoulli shift $G \times A^G \to A^G$ is defined by $(g,c) \mapsto g c$, 
where $(gc)(h) \coloneqq  c(g^{-1}h)$ for  $g,h \in G$ and $c \in A^G$. 
\par
Introduced by von Neumann \cite{neumann}, a \emph{cellular automaton} over  the group $G$ and the alphabet $A$ is a map
$\tau \colon A^G \to A^G$ admitting a finite \emph{memory set} $M \subset G$
and a \emph{local defining map} $\mu \colon A^M \to A$ such that 
\begin{equation*} 
\label{e:local-property}
(\tau(c))(g) = \mu((g^{-1} c )\vert_M)  \quad  \text{for all } c \in A^G \text{ and } g \in G.
\end{equation*} 
\par 
The well-known Gottschalk's conjecture \cite{gottschalk} asserts that over any universe, every injective cellular automaton with finite alphabet is surjective. We call a group $G$ \emph{surjunctive} if it satisfies Gottschalk's conjecture, i.e., for every finite alphabet $A$, every injective cellular automaton $\tau \colon A^G \to A^G$ must be surjective. Hence, it follows from Gromov-Weiss theorem (\cite{gromov-esav}, \cite{weiss-sgds}) that all sofic groups are surjunctive. 
\par 
Generalizing the direct finiteness conjecture, Kaplansky's stable finiteness conjecture for a field $k$ and a group $G$ states that the ring $\mathrm{Mat}_n(k[G])$ of square matrices of size $n$ with coefficients in $k[G]$ is directly finite for every $n \geq 1$. We establish the following general result  (Section~\ref{s:endo-ring-kaplansky}): 
\begin{theoremletter}
\label{t:intro-gottchalk-kaplansky} 
Kaplansky's stable finiteness conjecture holds for every surjunctive group $G$ and every field $k$. In particular,  Gottschalk's conjecture implies Kaplansky's stable finiteness conjecture. 
\end{theoremletter}
\par 
Let $G$ be a group and let $X$  be an  algebraic variety over a field $k$. Denote by $A=X(k)$ the set of rational points of $X$. Then the set $CA_{alg}(G,X,k)$ of \emph{algebraic cellular automata} consists of cellular automata $\tau \colon A^G \to A^G$ which  admit a memory  $M\subset G$ 
with local defining map $\mu \colon A^M \to A$ induced by some $k$-morphism of algebraic varieties 
$f \colon X^M \to X$, i.e., $\mu=f\vert_{A^M}$, 
where $X^M$ is the fibered product of copies of $X$ indexed by $M$. 
\par 
The central geometric result of the paper is the following direct finiteness property of the class of algebraic cellular automata $CA_{alg}$ (Section~\ref{s:proof-main-geometric}): 
 
\begin{theoremletter}
\label{t:main-ca-alg-direct}
Let $G$ be a surjunctive group.  Let $X$ be an algebraic variety over an   algebraically closed field $K$. Suppose that $\tau , \sigma \in CA_{alg}(G,X,K)$ satisfy  $\sigma \circ \tau = \Id$. Then one has  
$\tau \circ \sigma = \Id$. 
\end{theoremletter}
\par 
By \cite[Theorem~10.8]{phung-2020}, we have a  canonical isomorphism of near rings 
$R(k,G) \simeq CA_{alg}(G,\A^1,k)$ where $\A^1$ denotes the affine line and $k$ is an infinite field.  Consequently,  we can deduce an extension (Theorem~\ref{t:kap-near-ring-surjunctive-group}) of  Theorem~\ref{t:kap-near-ring-direct-finite} directly from Theorem~\ref{t:main-ca-alg-direct} where the group $G$ is only required to be surjunctive since we can suppose without loss of generality that $k$ is  algebraically closed. 
\par 
As another application of Theorem~\ref{t:main-ca-alg-direct}, we obtain (see Section~\ref{s:endo-ring-kaplansky}):
\begin{theoremletter}
\label{t:direct-group-endo-algr-intro}
Let $G$ be a surjunctive group. Let $R $ be the endomorphism ring of a commutative algebraic group over an  algebraically closed field. Then the group ring $R[G]$ is stably finite. 
\end{theoremletter}
\par 
Hence, we note that  Theorem~\ref{t:direct-group-endo-algr-intro} is a  generalization of the similar result   \cite[Corollary~1.3]{phung-2020} where the group $G$ is required to be sofic and $R$ is required to be the endomorphism ring of a connected commutative algebraic group.  
\par 
The paper is organized as follows. In Section~\ref{s:finite-data}, we recall basic facts in algebraic varieties and schemes of finite type. In Section~\ref{s:jacobson-schemes}, we formulate an extension of a lemma of Grothendieck to describe the set of closed points of relative fibered products of $\Z$-schemes of finite type. The proof of Theorem~\ref{t:main-ca-alg-direct} is given in Section~\ref{s:proof-main-geometric}. Finally, in  Section~\ref{s:application} we apply Theorem~\ref{t:main-ca-alg-direct} to obtain the proofs of other main results and their extensions  presented in the Introduction. 
\par

\section{Algebraic varieties and morphisms of finite type}
\label{s:finite-data}

\subsection{Algebraic varieties}  
Let $X$ be an algebraic variety  over a field $k$, i.e., a reduced separated $k$-scheme of finite type. We denote by  
$A\coloneqq X(k)$ the set of $k$-points of $X$. 
If the field $k$ is algebraically closed, we can identify $X$ with $A$. 
The following auxiliary result  states that morphisms of algebraic varieties are determined by their restrictions to the set of closed points. 

\begin{lemma}
\label{l:equalizer}
Let $k$ be an algebraically closed  field. 
Let $f, g \colon X \to Y$ be $k$-scheme morphisms of   $k$-algebraic varieties.  
Suppose that the restrictions  $f\vert_{X(k)}, g\vert_{X(k)} \colon X(k) \to Y(k)$ are equal. Then $f$ and $g$ are equal as morphisms of $k$-schemes. 
\end{lemma}

\begin{proof}
See, e.g.,  \cite[Lemma~7.2]{cscp-alg-ca}. 
\end{proof}
\par 
Recall also that an algebraic group is a group that is an algebraic variety with group operations
given by algebraic morphisms (cf.~\cite{milne}).

\subsection{Models of morphisms of finite type}
\label{s:model-finite-data} 

We shall need the following   auxiliary lemma in algebraic geometry: 
\begin{lemma}
\label{l:model-finite-data}
Let $n \in \N$ and let $X, Y$ be algebraic varieties over a field $k$. Let $f_i \colon X^n \to Y$, $i \in I$, be finitely many morphisms of $k$-algebraic varieties. Then there exist a finitely generated $\Z$-algebra $R \subset k$ and $R$-schemes of finite type $X_R$, $Y_R$ and $R$-morphisms $f_{i,R} \colon (X_R)^n \to Y_R$ of $R$-schemes with $X=X_R \otimes_R k$, $Y= Y_R \otimes_R k$, and $f_i=f_{i,R} \otimes_R k$ (base change to $k$). Moreover, if $X=Y$, one can take $X_R=Y_R$.
\end{lemma}

\begin{proof}
See, e.g., \cite[Section~8.8]{ega-4-3}, notably \cite[Scholie~8.8.3]{ega-4-3}, and  \cite[Proposition~8.9.1]{ega-4-3}. 
\end{proof}

\section{An extension of Grothendieck's lemma}  
\label{s:jacobson-schemes}

\subsection{Jacobson schemes}
We recall the notion of Jacobson schemes in \cite{ega-4-3}. 
Let $U$ be a topological space. A subset of $U$ which is the intersection of an open subset and a closed subset is said to be \emph{locally closed}. A \emph{constructible} subset of $U$ is a finite union of locally closed subsets of $U$. Observe that the complement of a constructible subset is also constructible. 
\par 
A scheme is called \emph{Jacobson scheme} if every nonempty constructible subset contains a closed point. For every field $k$, it follows from  \cite[Proposition~10.4.2]{ega-4-3} and \cite[Corollaire~10.4.6]{ega-4-3} that every $\Z$-scheme (resp. every $k$-scheme) of finite type is a Jacobson scheme. 
\par 
For every scheme $X$, we denote by $\kappa(x)$ the residual field at a point $x \in X$. Then we have the following result: 
\par 
\begin{lemma}
\label{l:t-p-d-stable}
Let $X$ be a scheme and let $Y$ be a Jacobson scheme. Suppose that $f \colon X \to Y$ is a morphism of finite type. Then $X$ is Jacobson and for every closed point $x$ of $X$, the image $f(x)$ is a closed point of $Y$. Moreover, we have a canonical inclusion $\kappa(f(x)) \subset \kappa(x)$ of fields.  
\end{lemma}

\begin{proof}
See \cite[Corollary~10.4.7]{ega-4-3}. The last statement is a basic property in scheme theory.
\end{proof}

\subsection{An extension of Grothendieck's lemma} 

To describe the set of closed points of a $\Z$-scheme of finite type, we recall the following lemma  due to Grothendieck:  

\begin{lemma}
\label{l:grothendieck-closed-point}
Let $X$ be a $\Z$-scheme of finite type. 
Then the set of closed points of $X$ is given by the following union of finite sets  $\cup_{p,d} T_{p,d}$ over all prime numbers $p \in \N$ and all $d \in \N$ where 
\begin{equation*}
    T_{p,d} = \{ x \in X \colon \vert \kappa (x) \vert = p^r, 1 \leq r \leq d\} \subset X\otimes_\Z \mathbb{F}_p. 
\end{equation*}
\end{lemma}

\begin{proof}
See \cite[Lemma~10.4.11.1]{ega-4-3}.  
\end{proof}

For each prime number $p$ and every $\Z$-scheme $V$, we denote by  $V_p=V \otimes_\Z \mathbb{F}_p$  the fibre of $V$ above the prime $p$. 
Let $X$ be an $S$-scheme and let $M$ be a finite set. To keep track of the factor, we denote by $X_S^M$ the $M$-fold fibered product $X\times_S \dots \times_S X$ of $S$-schemes indexed over the set $M$. 
\par 
We deduce from Lemma~\ref{l:grothendieck-closed-point} the following technical  consequence that we shall need in the proof of  Theorem~\ref{t:main-ca-alg-direct}:  

\begin{corollary}
\label{c:grothendieck-closed-point} 
Let $X$ be an $S$-scheme of finite type where $S$ is a $\Z$-scheme of finite type. Then for every finite set $M$, the set of closed points of $X_S^M$ is the union of finite sets  $\cup_{s,p,d} T_{s,p,d}^M$ over all prime numbers $p \in \N$, all closed points $s \in S_p$, and all $d \in \N$ where 
\begin{equation*}
    T_{p,s, d} = \{ x \in X_s \colon \vert \kappa (x) \vert = p^r, 1 \leq r \leq d\} \subset X_p,  
\end{equation*}
where $X_s= X \otimes_S \kappa(s)$ is the fibre of $X$ above $s$.  
\end{corollary}

\begin{proof}
First observe that $X$ and $X_S^M$ are $\Z$-schemes of finite type. Let us fix a prime number $p$ then we have  
\begin{equation} 
(X_S^M)_p= X_S^M \otimes_\Z \mathbb{F}_p 
=X_p\times_{S_p} \dots \times_{S_p} X_p= (X_p)_{S_p}^M. 
\end{equation}
\par 
Let $f \colon X \to S$ be the structural morphism. Then we infer from Lemma~\ref{l:t-p-d-stable} that every closed point of $(X_p)_{S_p}^M$ is given by an $M$-tuple $(x_g)_{g \in M}$ of closed points of $X_p$ such that $f(x_g)=s$ for all $g \in M$ for some closed point $s \in S_p$. The last condition is equivalent to the requirement that $x_g \in X_s$ for all $g \in M$ for some closed point $s \in S$. 
\par 
On the other hand,  Lemma~\ref{l:grothendieck-closed-point} implies the set of closed points of $X_p$ is given by the union over $d \in \N$ of the finite sets $ T_{p,d} = \{ x \in X_p \colon \vert \kappa (x) \vert = p^r, 1 \leq r \leq d\}$. 
From these descriptions, the conclusion of the corollary follows immediately. 
\end{proof}
\par

\section{Proof of the main geometric result}
\label{s:proof-main-geometric}

We are now in position to prove the main result of the paper Theorem~\ref{t:main-ca-alg-direct} which is a geometric direct finiteness property of symbolic algebraic varieties.

\begin{proof}[Proof of Theorem~\ref{t:main-ca-alg-direct}] 
Denote by  $A=X(K)$ the set of $K$-points of $X$. In particular, we can identify $X$ with the set  $A$. 
Let us fix $\tau , \sigma \in CA_{alg}(G,X,K)$ such that   $\sigma \circ \tau = \Id$. 
Let $\Gamma\coloneqq \tau(A^G) \subset A^G$ be the image of $\tau$.  
\par 
Now fix a finite subset  $\Omega \subset G$. 
Choose a finite subset $M \subset G$ large enough such that $M$ is a common memory set of both $\tau$ and  $\sigma$ and such that $1_G \in M$ and $M=M^{-1}$. Let 
$\mu \colon A^M \to A$ and $\eta \colon A^M \to A$ be respectively the algebraic local defining maps of $\tau$ and $\sigma$. 
\par 
Up to enlarging $M$, we can clearly suppose that $\Omega\subset M$. For every subset $E \subset G$, we denote $\Gamma_E=\{x\vert_E \colon x \in \Gamma\}\subset A^E$. We will show that  $\Gamma_M= A^M$. 
\par
Consider the algebraic $K$-morphism $\tau^+_M \colon A^{M^2} \to A^M$ of algebraic varieties given by 
$\tau^+_M(c)(g)= \mu((g^{-1}c)\vert_M)$ for every $c \in A^{M^2}$ and $g \in M$. Since $M$ is a memory set of $\tau$,  it is clear that we have  $\Gamma_M = \tau_M^+ (A^{M^2})$.   \par 
Since $\sigma \circ \tau= \Id$, it follows that for every $c \in A^{M^2}$, we have 
$\eta (\tau_M^+(c))=c(1_G)$.
Consequently, we infer from  Lemma~\ref{l:equalizer} that 
the composition morphism $\eta \circ \tau_M^+ \colon A^{M^2} \to A$ is exactly the canonical  projection morphism of $K$-schemes $\pi \colon A^{M^2} \to A^{\{1_G\}}$ induced by the inclusion $\{1_G\} \subset M^2$. 
\par 
Consider the finitely generated  $\Z$-sub-algebra 
$R\subset K$ given by Lemma~\ref{l:model-finite-data} applied to the morphisms $\mu, \eta \colon A^M \to A$. 
Consequently, we can find a corresponding $R$-scheme of finite type $A_R$ and morphisms $\mu_R, \eta_R \colon A_R^M \to A_R$ of $R$-schemes of finite type such that $A= A_R \otimes_R K$, $\mu= \mu_R \otimes_R K$, and $\eta=\eta_R \otimes_R K$. Here, we denote by $A_R^M = A_R \times_R \dots \times_R A_R$ the fibered product indexed by $M$ of the $R$-schemes $A_R$. 
\par 
For each $g \in G$, we have the 
canonical isomorphisms $A_R^{gM} \simeq A_R^M$ and $A_R^{\{g\}}\simeq A_R^{\{1_G\}}$ induced respectively by the trivial set bijections $M \simeq gM$, $h \mapsto gh$, and $ {\{1_G\}}\simeq {\{g\}}$.  Then 
similarly as above, we obtain an $R$-morphism of finite type  $\varphi \colon A_R^{M^2} \to A_R^M$ of $R$-schemes which is defined as a fibered product  over $g \in M$ of the morphisms $\mu_{g,R}\colon A_R^{gM} \to A_R^{\{g\}}$ induced by $\mu_R \colon A_R^M \to A_R$  via the canonical isomorphisms $A_R^{gM} \simeq A_R^M$ and $A_R^{\{g\}}\simeq A_R^{\{1_G\}}$. It follows immediately that  $\tau_M^+= \varphi \otimes_R K$. 
\par
Since $R \subset K$ and the composition map $\eta \circ \tau_M^+ \colon A^{M^2} \to A$ is exactly the canonical projection $A^{M^2} \to A^{\{1_G\}}$ as  $K$-morphisms of $K$-schemes, we can suppose, up to adding a finite number of generators to $R$, that the model morphism $\eta_R \circ \varphi \colon A_R^{M^2} \to A_R$ is  equal to the canonical projection $A_R^{M^2} \to A_R^{\{1_G\}}$ as $R$-morphisms of $R$-schemes (cf.~\cite[Scholie~8.8.3]{ega-4-3}). 
\par 
\begin{Claim} 
The morphism  $\varphi$ is surjective. 
\end{Claim}
\begin{proof}[Proof of the Claim]
Indeed, let $\PP\subset \N$ denote the set of prime numbers. For each $p \in \PP$, let us define $A_p = A \otimes_\Z \mathbb{F}_p$ the fibre of $A$ above the prime $p$. 
\par 
Since $A_R$ is an $R$-scheme of finite type and since $R$ is a $\Z$-algebra of finite type, the scheme $A_R$ is also a $\Z$-scheme of finite type. We deduce that $A_R$ and thus $A_R^M$ are Jacobson schemes. Therefore, we infer from Corollary~\ref{c:grothendieck-closed-point} that the set of closed points of $A_R^M$ is given by $\Delta = \cup_{p \in \PP, s \in S_p,  d\in \N} T_{p,d}^M$ where $S= \Spec R$ and $s \in S_p$ denotes a closed point and 
\begin{equation}
   T_{p,s, d}= \{x \in A_s \colon \vert \kappa(x) \vert=p^r, 1 \leq r \leq d\} 
\end{equation}
is a finite subset of $A_s$ and thus of $A_p$.  
\par 
Fix a prime $p \in \PP$ and a closed point $s \in S_p$. Let $T_{s} = \cup_{d \in \N} T_{p,s,d}$. We obtain by  restriction to the fibre $A_s \subset A$ above $s$ two cellular automata  
$\tau_s, \sigma_s \colon T_s^G \to T_s^G$ with  
$\tau_s= \tau\vert_{T_s^G}$ and $\sigma_s=\sigma\vert_{T_s^G}$. Note that the restrictions $\mu_s= \mu\vert_{T_s^M} \colon T_s^M \to T_s$ and $\eta_s= \eta\vert_{T_s^M} \colon T_s^M \to T_s$ are respectively well-defined local defining maps of $\tau_s$ and $\sigma_s$ by Lemma~\ref{l:t-p-d-stable}. 
\par 
Since $\eta_R \circ \varphi \colon A_R^{M^2} \to A_R$ is the canonical projection $A_R^{M^2} \to A_R^{\{1_G\}}$ as $R$-morphisms of $R$-schemes, 
the  restriction map $\left(\eta_R \circ \varphi \right)\vert_{T_s^{M^2}}$ is the same as the  canonical projection $T_s^{M^2} \to T_s^{\{1_G\}}$. It follows that $\sigma_s \circ \tau_s=\Id$. In particular, the cellular automaton $\tau_s$ is injective. 
\par 
Now let $d \in \N$. 
Since $\mu_{s}(T_{p,s,d}^M) \subset T_{p,s,d}$ by Lemma~\ref{l:t-p-d-stable}, we obtain a well-defined restriction cellular automaton 
$\tau_{p,s,d}=\tau_{s} \vert_{T_{p,s,d}^G} \colon T_{p,s,d}^G \to T_{p,s,d}^G$ admitting $\mu_{s}\vert_{T_{p,s,d}^M} \colon T_{p,s,d}^M \to T_{p,s,d}$ as a local defining map. As $\tau_s$ is injective, $\tau_{p,s,d}$ is also an injective cellular automaton. Since the alphabet $T_{p,s,d}$ is finite by Corollary~\ref{c:grothendieck-closed-point} and the group $G$ is surjunctive (see the Introduction), we can conclude that $\tau_{p,s,d}$ is surjective. 
\par 
Consequently, we deduce that $T_{p,s,d}^M \subset \varphi(A_R^{M^2})$ for all $p \in \PP$ and $d \in \N$. Therefore, it follows that 
$\Delta = \cup_{p \in \PP, s\in S_p, d \in \N} T_{p,s,d}^M \subset \varphi(A_R^{M^2})$. 
\par 
Observe on the one hand that $A_R^{M^2}$ and $A_R^M$ are Jacobson schemes and the image $\varphi (A_R^{M^2})$ is a constructible subset of $A_R^M$ by Chevalley's theorem (\cite[Th\' eor\` eme~1.8.4]{grothendieck-20-1964}). On the other hand, Corollary~ \ref{c:grothendieck-closed-point} tells us that $\Delta$ is   the set of all closed points of $A_R^M$. 
\par 
Therefore, we must have $\varphi (A_R^{M^2})= A_R^M$ by the definition of Jacobson schemes. The map $\varphi$ is thus surjective and the claim is proved. 
\end{proof}
\par 
 Note that surjectivity is a stable property under base change (see, e.g., \cite[Lemma~01S1]{stack}). Hence, it follows that $\tau_M^+= \varphi \otimes_R K$ is also a surjective morphism.  
Therefore, we find that 
\[ \Gamma_M= \tau_M^+(A^{M^2})= A^M.
\]
\par 
Since $\Omega \subset M$ by our choice of $M$, we deduce that $\Gamma_\Omega= A^\Omega$ for every finite subset $\Omega \subset G$. 
Therefore, $\tau(A^G)=\Gamma$ is dense in $A^G$ with respect to the prodiscrete topology. Since $\sigma \circ \tau= \Id$ by  hypothesis, $(\tau\circ \sigma)(x)=x$ for every $x \in \Gamma$. Thus, $(\tau\circ \sigma)(x)=x$ for all $x \in A^G$ as the prodiscrete topology on $A^G$ is Hausdorff and $\tau\circ \sigma$ is continuous since it is a cellular automaton (see \cite[Proposition~3.3]{cscp-alg-ca}).  Hence, $\tau \circ \sigma= \Id$ and the conclusion follows. 
\end{proof}
\par

\section{Applications} 
\label{s:application}

As applications of Theorem~\ref{t:main-ca-alg-direct}, we present in this section the   proofs of  Theorem~\ref{t:kap-near-ring-direct-finite},  Theorem~\ref{t:intro-gottchalk-kaplansky}, and Theorem~\ref{t:direct-group-endo-algr-intro} presented in the Introduction. We begin with the following generalization of Theorem~\ref{t:kap-near-ring-direct-finite}:

\begin{theorem}
\label{t:kap-near-ring-surjunctive-group}
Let $G$ be a surjunctive group and let $k$ be a field. Suppose that  
 $\alpha, \beta \in R(k,G)$ verify  $\alpha \star \beta= X_{1_G}$. Then $\beta \star \alpha=X_{1_G}$. 
\end{theorem}
 
 \begin{proof}
 We can trivially embed $k$ into some algebraically closed field $K$. In particular, we have an inclusion $R(k,G) \subset R(K,G)$. 
 Let $\A^1$ denote the affine line then we infer from  \cite[Theorem~10.8]{phung-2020} that 
there exists a canonical isomorphism of near rings 
$\Psi \colon R(K,G) \to CA_{alg}(G,\A^1,K)$ where the multiplication on $CA_{alg}(G,\A^1,K)$ is defined as the composition of maps. 
\par 
Therefore, we find that:  
\begin{equation}
\Psi(\alpha) \circ \Psi (\beta)= \Psi(\alpha \star \beta)= \Psi(X_{1_G})=\Id. 
\end{equation}
\par 
It follows from Theorem~\ref{t:main-ca-alg-direct} that $\Psi(\beta) \circ \Psi (\alpha)=\Id$. Hence, $\Psi(\beta \star \alpha)=\Id= \Psi(X_{1_G})$ and we deduce that $\beta \star \alpha=X_{1_G}$. The proof of the theorem is complete.
 \end{proof}

The proof of Theorem~\ref{t:intro-gottchalk-kaplansky} is very similar.

\begin{proof}[Proof of Theorem~\ref{t:intro-gottchalk-kaplansky}] 
We have to show that every surjunctive group satisfies Kaplansky’s stable finiteness conjecture. 
Let $k$ be a field and let $G$ be a surjunctive group. Let $n \geq 1$ and let 
$\alpha, \beta \in \mathrm{Mat}_n(k[G])$ be such that $\alpha \beta = 1$. We embed $k$ into an arbitrary algebraically closed field $K$. Then it follows trivially that $\alpha , \beta \in \mathrm{Mat}_n(K[G])$ and $\alpha \beta = 1$. 
We infer from \cite[Corollary~ 8.7.8]{ca-and-groups-springer} a ring isomorphism  
$\psi \colon \mathrm{Mat}_n(K[G])\to  LCA(G, K^n)$ 
 where $LCA(G, K^n)$ is the $K$-algebra of cellular automata $K^G \to K^G$ admitting $K$-linear local defining maps. The multiplication on $LCA(G, K^n)$ is given by the composition of maps.  Therefore,  $\psi(\alpha)\circ \psi(\beta)=\psi(\alpha\beta)=\psi(1)=\Id$. 
 \par 
As  $LCA(G,K^n)\subset CA_{alg}(G,\A^n, K)$, where $\A^n$ is the $n$-dimensional affine space, and as $G$ is surjunctive, Theorem~\ref{t:main-ca-alg-direct} implies that $\psi(\beta)\circ \psi(\alpha)=\Id$ and therefore $\psi(\beta \alpha)=\Id$. It follows that $\beta \alpha= 1$ since $\psi$ is an isomorphism. The proof is thus complete. 
\end{proof}

\subsection{Symbolic group varieties}
\label{s:endo-ring-kaplansky}
Let $G$ be a group and let $k$ be a field. Let $X$ be an algebraic group over $k$ and let $A=X(k)$. 
The set $CA_{algr}(G,X,k)$ of \emph{algebraic group cellular automata} consists of cellular automata $\tau \colon A^G \to A^G$ which admit a memory set $M$ 
with local defining map $\mu \colon A^M \to A$ induced by some homomorphism of algebraic groups
$f \colon X^M \to X$, i.e., $\mu=f\vert_{A^M}$. It is clear that we have an inclusion $CA_{algr}(G,X,k) \subset CA_{alg}(G,X,k)$. See also \cite{phung-dual-surjunctivity}, \cite{phung-dcds}, \cite{phung-israel}, or \cite{cscp-invariant-ca-alg} for more details. 

\begin{proof}[Proof of Theorem~\ref{t:direct-group-endo-algr-intro}] 
Let $Y$ be a commutative  algebraic group over an  algebraically closed field $K$. Let $R=\End_K(Y)$ be the endomorphism ring (of $K$-homomorphisms of algebraic groups) of $Y$. Let $G$ be a surjunctive group. Let $n \in \N$ and let $X=Y^n$. Then  $\mathrm{Mat}_n(R)=\End_K(X)$  by \cite[Lemma~9.5]{phung-2020}. 
By \cite[Proposition~9.3]{phung-2020}, there exists a natural  ring isomorphism $\End_K(X)[G] \simeq CA_{algr}(G,X,K)$. We infer from Theorem~\ref{t:main-ca-alg-direct} that $CA_{algr}(G,X,K)$ is directly finite as a subspace of $CA_{alg}(G,X,K)$. It follows that the ring $\mathrm{Mat}_n(R)[G]=\End_K(X)[G]$ is also directly finite. Finally, since we have an  isomorphism  $\mathrm{Mat}_n(R[G])\simeq \mathrm{Mat}_n(R)[G]$ by \cite[Lemma~9.4]{phung-2020}, the ring $\mathrm{Mat}_n(R[G])$ is also directly finite and the conclusion follows.  
\end{proof}

\bibliographystyle{siam}

\end{document}